\documentclass[12pt]{amsart}

\usepackage{amssymb, amsthm, amsmath}
\usepackage{amsrefs}

\newtheorem{theorem}{Theorem}[section]
\newtheorem{lemma}[theorem]{Lemma}

\theoremstyle{definition}

\theoremstyle{remark}
\newtheorem{remark}[theorem]{Remark}

\numberwithin{equation}{section}

\newcommand{\D}{\mathbb{D}} 
\newcommand{\cD}{\overline{\D}}

\newcommand{\C}{\mathbb{C}}
\newcommand{\T}{\mathbb{T}}

\title{Uchiyama's lemma and the John-Nirenberg inequality}
\author{Greg Knese}
\date{\today}
\thanks{Research supported by National Science Foundation grant DMS 1048775}
\address{University of Alabama, Dept of Mathematics, Tuscaloosa, AL,
  35487-0350}
\email{geknese@bama.ua.edu}
\keywords{John-Nirenberg, Uchiyama, BMO, BMOA, Hardy spaces}
\subjclass[2010]{30H35, 30H10, 30J99} 
\urladdr{http://bama.ua.edu/~geknese}

\begin{document}

\begin{abstract}
Using integral formulas based on Green's theorem and in particular a
lemma of Uchiyama, we give simple proofs of comparisons of different
BMO norms without using the John-Nirenberg inequality while we also
give a simple proof of the strong John-Nirenberg inequality. Along the
way we prove $BMOA \subset (H^1)^*$ and $BMO \subset
\text{Re}(H^1)^*$.
\end{abstract}

\maketitle

\section{Introduction}
The space of functions of bounded mean oscillation(BMO), initially
introduced in the study of PDEs, is most famously known from the
Fefferman duality theorem as the dual of the real Hardy space
$\text{Re}(H^1)$ \cite{cF71}.  The John-Nirenberg inequality is the
traditional point of entry for understanding BMO \cite{JN61}.  BMO on
the unit circle $\T$ is most naturally defined using the norm
\[
||f||_{*} = \sup_{I \subset \T} \frac{1}{|I|} \int_{I} |f - f_I| ds
\]
where the supremum is over intervals $I \subset \T$ and $f_I =
\frac{1}{|I|} \int_{I} fds $. A function $f\in L^1$ is then in $BMO$
if the above supremum is finite. (Here $ds$ is arc length measure.)

It turns out to be useful to use two other norms on BMO.  Another norm
is obtained by using the normalized Poisson kernel $P_z(\zeta) =
\frac{1-|z|^2}{2\pi|1-\bar{z}\zeta|^2}$ to perform averaging
\[
||f||_{BMO_1} = \sup_{z \in \D} \int_{\T} |f - f(z)|P_z ds
\]
where we use the harmonic extension of $f$, $f(z) = \int_{\T} f
P_z$. The proof of equivalence is obtained by comparing $P_z$ to
appropriate box kernels $\frac{1}{|I|} \chi_I$.  (See Garnett
\cite{jG07} Chapter 6, Section 1.) Yet another norm is the Garsia norm
\[
||f||^2_{BMO_2} := \sup_{z\in \D} \int_{\T} |f - f(z)|^2P_zds = \sup_{z \in \D}
[(|f|^2)(z) - |f(z)|^2]
\]
where $|f|^2(z)$ and $f(z)$ denote values of harmonic extensions of
$|f|^2$ and $f$ respectively. For this definition we need $f \in L^2$.

Why are there so many norms? The Garsia norm is the easiest norm to
use when proving that BMO is the dual of the real Hardy space
$\text{Re}(H^1)$, but the norm $\|\cdot\|_{*}$ and the norm
$\|\cdot\|_{BMO_1}$ most exemplify the phrase ``bounded mean
oscillation.'' Unfortunately, it is not obvious that the Garsia norm
is equivalent to the earlier norms and indeed this is one of the main
purposes of the John-Nirenberg inequality.  The John-Nirenberg
inequality says there exist constants $c, C>0$ such that for any
interval $I \subset \T$
\[
\frac{|\{\zeta \in I: |f(\zeta)-f_I| > \lambda\}|}{|I|} \leq
C\exp\left(\frac{-c\lambda}{\|f\|_{*}} \right).
\]
This statement is implied by the \emph{strong} John-Nirenberg
inequality: there exists $c>0$ such that $\epsilon < c/\|f\|_{*}$
implies
\[
\sup_{I \subset \T} \frac{1}{|I|} \int_{I} e^{\epsilon |f-f_I|} ds < \infty.
\]
For more background and the traditional approach to all of this
material, the reader should consult Garnett \cite{jG07} chapter 6.  

In this paper we somewhat turn things around by proving the
equivalence of the norms $\|\cdot\|_{BMO_1}$ and $\|\cdot\|_{BMO_2}$
without using John-Nirenberg and then prove a strong John-Nirenberg
inequality in terms of the norm $\|\cdot\|_{BMO_2}$.  All proofs of
the John-Nirenberg inequality, of which we are aware, involve some
kind of Calderon-Zygmund decomposition and a stopping-time
argument. (More sophisticated variants of these ideas have been
employed in finding sharp versions of the John-Nirenberg
inequality. See \cite{aK90}, \cite{SV07}, and \cite{SV11}.)  In
contrast, the proof presented in this article uses only Green's
theorem and most importantly Uchiyama's lemma.  This approach owes a
great debt to several recent approaches to traditionally difficult
theorems in complex analysis beginning with the corona theorem as
proved by T.Wolff (see \cite{jG07} Chapter 8), the
Hunt-Muckenhoupt-Wheeden theorem as proved in \cite{NT96}, and the
reproducing kernel thesis for Carleson measures as proved in
\cite{PTW07}.  The book by Andersson \cite{mA97} features many aspects
of the present approach as well.

\subsection*{Acknowledgments} 
Thanks to John M$^{\text{c}}$Carthy, Kabe Moen, and Tavan Trent for useful
comments and conversations.

\section{Main Results}

For definiteness, we shall say a real valued function $f\in L^2(\T)$
modulo constant functions is in $BMO$ if the norm $\|f\|_{BMO_2}$ as
above is finite.  An analytic function $f$ in the Hardy space
$H^2(\T)$ modulo the constant functions is in $BMOA$ if $\|f\|_{BMO_2}$
as above is finite.

The John-Nirenberg inequality is typically required to prove the norms
$\|f\|_{BMO_1}$ and $\|f\|_{BMO_2}$ are equivalent (or even to show that $f
\in L^2$ with finite $\|\cdot\|_{BMO_2}$ norm is in $L^1$).  We get
around this fact in the setting of BMOA, and we are able to prove the
following theorem without using the John-Nirenberg theorem.  

\begin{theorem} \label{normcomparison} 
If $F \in BMOA$, then
\[
||F||_{BMO_1} \leq ||F||_{BMO_2} \leq 2\sqrt{e} ||F||_{BMO_1}.
\]
\end{theorem}

The first inequality is just Cauchy-Schwarz.  Certain aspects of our
approach become more technical in the case of real $BMO$.
Nevertheless, we are still able to prove the following comparison
without John-Nirenberg.

\begin{theorem} \label{normcomparisonharmonic}
If $u \in BMO$, then
\[
\|u\|_{BMO_1} \leq \|u\|_{BMO_2} \leq (21) \|u\|_{BMO_1}.
\]
\end{theorem}

The explanation for the non-sharp constant $21$ will have to wait
until Section \ref{sec:harmnormcomp}.
Finally, we prove the following version of the strong John-Nirenberg
inequality.

\begin{theorem} \label{johnnirenberg}
Let $F \in BMOA$.  For any $\epsilon < \frac{2}{\sqrt{e} \|F\|_{BMO_2}}$, we have
\[
\int_{\T} e^{\epsilon |F-F(z)|} P_z ds < \frac{3}{(1-\frac{\epsilon
    \sqrt{e}}{2} \|F\|_{BMO_2})^{3/2}}.
\]
\end{theorem}

Let us point out a couple of direct consequences.  If $u \in L^2(\T)$
is harmonically extended into the unit disk $\D$ and $F =
u+i\tilde{u}$, where $\tilde{u}$ is the harmonic conjugate of $u$,
then we can prove
\[
\int_{\T} e^{\epsilon |u-u(z)|} P_z ds < \frac{3}{(1-\epsilon
    \sqrt{2e} \|u\|_{BMO_2})^{3/2}}
\]
using the fact $2\|u\|^2_{BMO_2} = \|F\|^2_{BMO_2}$ (Remark
\ref{remark:uf}).

Using Theorem \ref{normcomparison}, we also have
\[
\int_{\T} e^{\epsilon |F-F(z)|} P_z ds < \frac{3}{(1-\epsilon e
  \|F\|_{BMO_1})^{3/2}}
\]
which shows this integral is finite so long as $\epsilon <
1/(e\|F\|_{BMO_1})$.

\section{Definitions, Green's theorem, and Hardy-Stein identities}
We use $ds$ to denote arc length measure on the unit circle $\T$ or
the circle $r\T$ of radius $r$.  The measure $dA$ denotes area measure
in the complex plane $\C$, and $\D$ and $r\D$ refer to the open unit disk
and the disk of radius $r$, respectively.  We use the following notations
\[
\partial = \frac{1}{2}(\frac{\partial}{\partial x}-i
\frac{\partial}{\partial y}),  \bar{\partial} = \frac{1}{2}(\frac{\partial}{\partial x}+i
\frac{\partial}{\partial y}), \Delta = 4 \partial \bar{\partial}.
\] 

One form of Green's theorem is
\[
\int_{r\T} f P^{(r)}_z ds - f(z) = \int_{r\D} \Delta f g^{(r)}_z dA
\]
where $|z|<r\leq 1$, 
\[
g^{(r)}_z(\zeta) = \frac{1}{2\pi} \log \left|\frac{z-\zeta}{r-\frac{\bar{z}\zeta}{r}}\right|^{-1}
\text{ and  }
P^{(r)}_z(\zeta)  =
\frac{1-|z/r|^2}{2\pi |1-\bar{z}\zeta/r^2|^2}, 
\]
assuming $f \in C^2(r\cD)$. Write $P_z = P^{(1)}_z$ and $g_z =
g^{(1)}_{z}$.  It is worth noting that
\[
g^{(r)}_z \nearrow g_z 
\]
as $r \nearrow 1$, and 
\[
|P^{(r)}_z(r\zeta)-P_z(\zeta)|\to 0
\]
uniformly for $\zeta \in \T$ as $r \nearrow 1$.

Applying Green's theorem to $(|f|^2+\epsilon)^{p/2}$ or
$(u^2+\epsilon)^{p/2}$ and carefully making sure we can send $\epsilon
\searrow 0$ and then $r\nearrow 1$, is one way to prove the following
classical Hardy-Stein identities. See \cite{mP09}.

\begin{lemma} \label{lem:p}
For $f \in H^p(\D)$, $0<p < \infty$, $z \in \D$
\[
\int_{\T} |f|^p P_z ds - |f(z)|^p = \iint_{\D} p^2 |\partial f|^2|f|^{p-2} g_z dA.
\]
For $1<p<\infty$, $u \in L^p(\T)$ extended harmonically into $\D$, and
$z \in \D$ we have
\[
\int_{\T} |u|^p P_z ds -|u(z)|^p = p(p-1) \iint_{\D} |u|^{p-2} |\nabla u|^2 g_z dA.
\] 
\end{lemma}

Here $H^p(\D)$ is the Hardy space on the unit disk with exponent $p$.

\begin{remark} \label{remark:uf}
Notice that 
\[
\int_{\T} |f-f(z)|^2 P_z ds = 4 \iint_{\D} |\partial f|^2 g_z dA.
\]
If $u = \text{Re}(f)$ then $2|\partial f|^2 = |\nabla u|^2$ and we
see that
\[
\int_{\T} |f-f(z)|^2 P_z ds = 2\int_{\T} (u-u(z))^2 P_z ds
\]
which implies $\|f\|^2_{BMO_2} = 2 \|u\|^2_{BMO_2}$, a fact we use several
times.
\end{remark}

\section{Uchiyama's lemma}
Uchiyama's lemma is our most important tool. (See Nikolskii \cite{nN86}
page 290 and the notes on page 296.)

\begin{lemma} \label{uchiyama} 
If $\phi \in C^2(\D)$ and $f$ is holomorphic in $\D$, then for $0\leq
|z|<r<1$
\[
\int_{r\T} |f| e^{\phi} P^{(r)}_z ds \geq \iint_{r\D} \Delta \phi e^\phi |f| g^{(r)}_z  dA
\]
\end{lemma}

\begin{proof}
For any $\psi \in C^2(\D)$
\[
\Delta e^\psi = e^\psi (\Delta \psi + |\nabla \psi|^2) \geq e^\psi
\Delta \psi.
\]
Applying this to $\psi = \phi + \frac{1}{2}\log(|f|^2+\epsilon)$
we have
\[
\Delta (e^\phi (|f|^2+\epsilon)^{1/2}) \geq e^\phi
(|f|^2+\epsilon)^{1/2} \Delta \phi
\]
after noticing that
\[
\Delta \frac{1}{2}\log (|f|^2+\epsilon) = \frac{2\epsilon |\partial
  f|^2}{(|f|^2+\epsilon)^2} \geq 0.
\]

By Green's theorem,
\[
\begin{aligned}
& \int_{r\T} e^\phi (|f|^2+\epsilon)^{1/2} P^{(r)}_z ds -
e^{\phi(z)}(|f(z)|^2+\epsilon)^{1/2} \\
&= \iint_{r\D} \Delta(e^{\psi}) g^{(r)}_z dA \\
& \geq \iint_{r\D} e^\phi (|f|^2+\epsilon)^{1/2} \Delta \phi g^{(r)}_z dA 
\end{aligned}
\]
Letting $\epsilon \to 0$, we get
\[
\int_{r\T} e^{\phi} |f| P^{(r)}_z ds \geq \iint_{r\D} e^{\phi} |f| \Delta \phi g^{(r)}_z dA.
\]
\end{proof}

\begin{lemma} \label{lem:mainlemma}

Let $F \in BMOA, f \in H^1$. Then,
\[
 \iint_{\D} |\partial F|^2 |f|
g_z dA \leq \frac{e}{4} ||f P_z||_{L^1} ||F||^2_{BMO_2}
\]
Let $u \in BMO, f \in H^1$.  Then,
\[
\iint_{\D} |\nabla u|^2 |f| g_z dA \leq \frac{e}{2} \|f P_z \|_{L^1} \|u\|^2_{BMO_2}
\]
\end{lemma}

\begin{proof} Let $|z|< r < 1$.  Apply the previous lemma to $\phi = (|F(\zeta)|^2 -
  (|F|^2)(\zeta))/||F||^2_{BMO_2}$ which is non-positive, bounded below
  by $-1$, and subharmonic since
\[
\Delta \phi = 4 (|\partial F|^2) /||F||^2_{BMO_2}.
\]
We arrive at
\[
\begin{aligned}
\int_{r\T} |f| P^{(r)}_z ds &\geq \int_{r\T} e^{\phi} |f| P^{(r)}_z ds \\
&\geq \iint_{r\D} e^{\phi} \frac{4}{\|F\|^2_{BMO_2}}|\partial F|^2 |f| g^{(r)}_z dA \\
& \geq \frac{4e^{-1}}{\|F\|^2_{BMO_2}} \iint_{r\D} |\partial F|^2 |f| g^{(r)}_z dA
\end{aligned}
\]
After doing so let $r \to 1$ and the first part of the lemma is
proved. 

For the second part, set $\phi = ((u(\zeta))^2 -
u^2(\zeta))/\|u\|^2_{BMO_2} \geq -1$ and 
notice that
\[
\Delta \phi = 2\frac{|\nabla u|^2}{\|u\|^2_{BMO_2}}.
\]
Then,
\[
\int_{r\T} |f| P^{(r)}_{z} ds \geq \frac{2e^{-1}}{\|u\|^2_{BMO_2}}
\iint_{r\D} |\nabla u |^2 |f| g^{(r)}_z dA.
\]
After letting $r\to 1$, the second inequality follows.
\end{proof}

\section{Theorem \ref{normcomparison}: norm comparison for BMOA}

Along the way to proving Theorem \ref{normcomparison} (the norm
comparison for BMOA), it is useful to prove the key estimate for
proving $BMOA \subset (H^1)^*$ or $BMO \subset \text{Re}(H^1)^*$.  The
approach is similar to Andersson \cite{mA97} (see chapters 8 and 9)
with a stricter accounting of the constants involved.

\begin{theorem} \label{mainthm} Let $F \in BMOA, h \in H^2$.  Then,
\[
\left| \int_{\T} F \bar{h} P_z ds \right| \leq 2\sqrt{e} ||F||_{BMO_2} ||h
P_z||_{L^1}.
\]
Let $u \in BMO, h \in H^2$. Then,
\[
\left| \int_{\T}  u \text{Re}(h) P_z ds \right| \leq \sqrt{2e} ||u||_{BMO_2}
||h P_z||_{L^1}.
\]
\end{theorem}

The reason for having $h \in H^2$ as opposed to $H^1$ is that the
integrals may not converge absolutely for $h\in H^1$. However, since
$H^2$ is dense in $H^1$, the estimates imply that integration against
a function in $BMOA$ extends to a bounded linear functional on $H^1$.

\begin{proof}[Proof of Theorem \ref{mainthm}]

We may assume $F(z)=0$ since $F$ is only a function modulo constant
functions.  By Green's theorem (or a polarized Hardy-Stein identity
for $p=2$),
\[
\int_{\T} F \bar{h} P_z ds = 4 \iint_{\D} \partial F \overline{\partial h} g_z dA
\]
By Cauchy-Schwarz, in modulus this is less than or equal to
\[
\begin{aligned}
&4 \left( \iint_{\D} |\partial F|^2 |h| g_z dA
\right)^{1/2} \left( \iint_{\D} \frac{|\partial h|^2}{|h|} g_z dA \right)^{1/2} \\
& \leq 4 \left(\frac{e}{4} ||h P_z ||_{L^1} ||F||^2_{BMO_2} \right)^{1/2}
||h P_z||^{1/2}_{L^1} \\
& \leq 2\sqrt{e} ||h P_z||_{L^1} ||F||_{BMO_2}
\end{aligned}
\]
where the first inequality follows from Lemmas \ref{lem:mainlemma} and
\ref{lem:p} (with $p=1$).

Similarly,
\[
\int_{\T} u \text{Re}(h) P_z ds = 4 \text{Re} \iint_{\D} \partial h
\bar{\partial} u g_z dA
\]
and again by Cauchy-Schwarz this is bounded by
\[
\begin{aligned}
&2 \left( \iint_{\D} |\nabla u|^2 |h| g_z dA
\right)^{1/2} \left( \iint_{\D} \frac{|\partial h|^2}{|h|} g_z dA \right)^{1/2} \\
& \leq \sqrt{2e} ||h P_z||_{L^1} ||u||_{BMO_2}
\end{aligned}
\]
(after being careful with using $4|\partial u|^2 = |\nabla u|^2$).

\end{proof}

\begin{remark} \label{slicesremark}
If $f\in L^2(\T)$ or $f\in L^1(\T)$ is harmonically extended into
$\D$, then defining $f_r(\zeta) := f(r\zeta)$ we have
\[
\|f_r\|_{BMO_j} \nearrow \|f\|_{BMO_j} \text{ as } r \nearrow 1 \text{
  for } j =1,2
\]
and this holds even if one of the norms is infinite. (We leave the
proof of this fact to the reader.)  Because of this it suffices to
prove Theorem \ref{normcomparison} for $F_r$ or Theorem
\ref{normcomparisonharmonic} for $u_r$.

Indeed, if we assume that $f \in L^1(\T)$ and $\|f\|_{BMO_1} <
\infty$, and if we have proven
\[
\|f_r\|_{BMO_2} \leq C\|f_r\|_{BMO_1}, 
\]
 then in particular $\sup_{0<r<1} \int_{\T}|f_r|^2ds < \infty$ and so
 $f\in L^2(\T)$ by standard approximate identity properties for the
 Poisson kernel. It then follows that $\|f\|_{BMO_2} \leq
 C\|f\|_{BMO_1}$.
\end{remark}


\begin{proof}[Proof of Theorem \ref{normcomparison}:]
As remarked above, we can replace $F$ with $F_r$.  In this case,
Theorem \ref{normcomparison} is an immediate corollary of Theorem
\ref{mainthm} if we now replace $F$ with $F-F(z)$ and $h$ with
$F-F(z)$ in the statement of Theorem \ref{mainthm}. This gives
\[
\int_{\T} |F - F(z)|^2 P_zds \leq 2\sqrt{e} ||F||_{BMO_2} \int_{\T}
|F-F(z)| P_z ds
 \]
and taking a supremum over $z$ yields Theorem \ref{normcomparison}.
\end{proof}

Theorem \ref{normcomparisonharmonic} (the norm comparison for real
BMO) is seemingly not so easy to deduce from Theorem \ref{mainthm} as
the best it gives is the estimate
\[
\|u\|_{BMO_2} \leq \sqrt{2e} \|u+i\tilde{u}\|_{BMO_1}
\]
where $\tilde{u}$ is the harmonic conjugate of $u$.  As there is no
direct comparison of $u$ and $\tilde{u}$ in terms of $L^1$ norms
(unlike in the $L^2$ situation), it seems the BMO condition needs to
play a more active role in the comparison of $\|u\|_{BMO_1}$ and
$\|u\|_{BMO_2}$.  One of the lemmas in the proof of the strong
John-Nirenberg inequality is used in proving Theorem
\ref{normcomparisonharmonic}, so we postpone the proof to Section
\ref{sec:harmnormcomp}.

\section{Theorem \ref{johnnirenberg}: The strong John-Nirenberg inequality}

\begin{lemma} \label{HS-k}
For $F\in H^k$,
\[
\int_{\T} |F-F(z)|^{k} P_zds =  k^2 \iint_{\D} |\partial F|^2
|F-F(z)|^{k-2} g_z dA
\]
\end{lemma}

\begin{proof}
This is lemma \ref{lem:p} with $f = F-F(z)$ and $p=k$.
\end{proof}

\begin{lemma} \label{inductlemma}
For $F \in BMOA$, 
\[
\int_{\T} |F-F(z)|^{k} P_zds \leq e (k/2)^2 \int_{\T} |F-F(z)|^{k-2} P_zds \| F\|_{BMO_2}^2
\]
so that inductively we have
\[
\int_{\T} |F-F(z)|^{2k} P_zds \leq e^k k!^2 \| F\|^{2k}_{BMO_2}
\]
and
\[
\int_{\T} |F-F(z)|^{2k+1} P_zds  
\leq (e/4)^k
\left(\frac{(2k+1)!}{2^kk!}\right)^2 \|F\|^{2k+1}_{BMO_2}. 
\]

\end{lemma}

\begin{proof}
If we apply Lemma \ref{lem:mainlemma} with $F \in BMOA$ and
$f=(F-F(z))^{k-2}$, then
\[
\iint_{\D} |\partial F|^2 |F-F(z)|^{k-2} g_zdA \leq (e/4) \int_{\T}
|F-F(z)|^{k-2} P_zds \| F\|_{BMO_2}^2.
\]
Coupled with Lemma \ref{HS-k},
\[
\int_{\T} |F-F(z)|^{k} P_z ds \leq (e/4) k^2 \int_{\T} |F-F(z)|^{k-2} P_z ds\| F\|_{BMO_2}^2
\]
and the rest follows by iterating this inequality. 
\end{proof}

\begin{proof}[Proof of Theorem \ref{johnnirenberg}]
Observe that by Lemma \ref{inductlemma}
\[
\begin{aligned}
&\int_{\T} e^{\epsilon |F-F(z)|} P_zds = \sum_{k\geq 0} \frac{\epsilon^k}{k!}
\int_{\T} |F-F(z)|^k P_zds \\ 
&= \sum_{k \geq 0} \frac{\epsilon^{2k}}{(2k)!}
\int_{\T} |F-F(z)|^{2k} P_zds + \frac{\epsilon^{2k+1}}{(2k+1)!} \int_{\T}
|F-F(z)|^{2k+1}P_z ds \\ 
& \leq \sum_{k \geq 0} \frac{\epsilon^{2k}}{(2k)!}
e^k(k!)^2 \|F\|^{2k}_{BMO_2} + \frac{\epsilon^{2k+1}}{(2k+1)!} (e/4)^k
\left(\frac{(2k+1)!}{2^kk!}\right)^2 \|F\|^{2k+1}_{BMO_2}\\
&=\left(\sum_{k\geq 0} \frac{(k!)^2}{(2k)!}(2x)^{2k}\right) +
\left(\frac{2}{\sqrt{e}}x\sum_{k\geq 0} \frac{(2k+1)!}{4^k(k!)^2}
x^{2k}\right)
\end{aligned}
\]
where $x = (1/2)\epsilon\sqrt{e} \|F\|_{BMO_2}$.  The last expression
can be explicitly computed. Whenever $x < 1$ it is equal to
\[
\frac{1}{1-x^2} + \frac{x\arcsin(x) + \frac{2}{\sqrt{e}}
  x}{(1-x^2)^{3/2}} \leq
\frac{\frac{\pi}{2}+\frac{2}{\sqrt{e}}}{(1-x)^{3/2}} 
\]
and since $\frac{\pi}{2}+\frac{2}{\sqrt{e}} < 3$
\[
\int_{\T} e^{\epsilon |F-F(z)|} P_z ds \leq \frac{3}{(1-\frac{\epsilon
    \sqrt{e}}{2} \|F\|_{BMO_2})^{3/2}}.
\]

\end{proof}

\section{Theorem \ref{normcomparisonharmonic}, norm
  comparison for real BMO} \label{sec:harmnormcomp}
The Hardy-Stein identity for harmonic functions fails for $p=1$ and so
we do not have a nice Green's theorem formula for the expression
\[
\int_{\T} |u-u(z)|P_z ds.
\]
A replacement is in the following lemma.

\begin{lemma} \label{h1lemma} If $u \in L^1(\T)$ (and extended
  harmonically into $\D$), then
\[
\iint_{\D} \frac{|\nabla u|^2}{((u-u(z))^2+1)^{3/2}} g_z dA \leq
\int_{\T}|u-u(z)|P_z ds 
\]
\end{lemma}

\begin{proof}
By Green's theorem, for $|z|<r<1$
\begin{equation} \label{h1eq}
\iint_{r\D} \frac{|\nabla
  u|^2}{((u-u(z))^2+1)^{3/2}} g^{(r)}_z dA = \int_{r\T}((u-u(z))^2+1)^{1/2}P^{(r)}_zds - 1.
\end{equation}
Setting $v = u-u(z)$, this follows from
\[
\Delta (v^2+1)^{1/2} 
=|\nabla v|^2(v^2+1)^{-3/2}.
\]
Since $u_r$ converges to $u$ in $L^1(\T)$ as
$r\nearrow 1$ (recall this only uses basic approximate identity
properties of the Poisson kernel), it can be shown that 
\[
\lim_{r\nearrow 1} \int_{r\T} ((u-u(z))^2+1)^{1/2} P^{(r)}_z ds =
\int_{\T} ((u-u(z))^2+1)^{1/2} P_z ds
\]
since $|\sqrt{x^2+1}-\sqrt{y^2+1}| \leq C|x-y|$. On the other hand,
the left hand side of \eqref{h1eq} converges monotonically to 
\[
\iint_{\D} \frac{|\nabla u|^2}{((u-u(z))^2+1)^{3/2}} g_z dA =
\int_{\T} ((u-u(z))^2+1)^{1/2} P_z ds - 1.
\]
The desired inequality 
\[
\int_{\T} ((u-u(z))^2+1)^{1/2} P_z ds - 1 \leq \int_{\T}|u-u(z)|P_z ds
\]
follows from the inequality $\sqrt{1+x^2} \leq 1+|x|$.
\end{proof}

Theorem \ref{normcomparisonharmonic} follows from the next result
since $e^{2/3}5^{1/3}3^{5/3} \leq 21$.

\begin{theorem} Let $u \in L^1(\T)$ with $\|u\|_{BMO_1} < \infty$. Then,
\[
\|u\|_{BMO_2} \leq e^{2/3}5^{1/3}3^{5/3} \|u\|_{BMO_1}
\]
\end{theorem}

\begin{proof}
As in Remark \ref{slicesremark}, it is enough to prove the theorem
with $u$ replaced with $u_r$. Observe
\[
\begin{aligned}
&\int_{\T} (u-u(z))^2P_z ds  = 2 \iint_{\D} |\nabla u|^2 g_z dA \\
&\leq 
 2 \sqrt{\iint_{\D} \frac{|\nabla u|^2}{((u-u(z))^2+1)^{3/2}} g_z
  dA} \sqrt{\iint_{\D} |\nabla u|^2 ((u-u(z))^2+1)^{3/2} g_z dA} \\
&\overset{\text{Lemma \ref{h1lemma}}}{\leq}  2 \sqrt{\int_{\T}|u-u(z)|P_z ds} \sqrt{\iint_{\D}
  |\nabla u|^2\sqrt{2}(1+|u-u(z)|^3) g_z dA} \\
&\leq 2^{5/4}\sqrt{\|u\|_{BMO_1}} \sqrt{A+B}.
\end{aligned}
\]
The third line uses the inequality $(1+x^2)^{3/2} \leq
\sqrt{2}(1+|x|^{3})$ and in the last line we define and estimate two
quantities $A$ and $B$
\[
A:= \iint_{\D} |\nabla u|^2 g_z dA = \frac{1}{2}\int_{\T} (u-u(z))^2P_z ds
\leq \frac{1}{2} \|u\|^2_{BMO_2}
\]
and by Lemma \ref{lem:p}
\[
B:= \iint_{\D} |\nabla u|^2|u-u(z)|^3 g_z dA = \frac{1}{20} \int_{\T}
|u-u(z)|^5P_z ds.
\]
We can estimate $B$ by letting $f = u+i\tilde{u}$, where $\tilde{u}$
is the harmonic conjugate of $u$, and by using the inequalities for
holomorphic functions that have already been established.  Namely,
\[
\begin{aligned}
\int_{\T}
|u-u(z)|^5P_z ds &\leq \int_{\T} |f-f(z)|^5 P_z ds \\
&\leq
(\frac{15e}{4})^2 \|f\|^5_{BMO_2} 
= \frac{(15e)^2\sqrt{2}}{4} \|u\|^5_{BMO_2}
\end{aligned}
\]
by Lemma \ref{inductlemma} and the fact that $\|f\|^2_{BMO_2} =
2\|u\|^2_{BMO_2}$. Stringing everything together and taking a supremum
over $z \in \D$ gives
\[
\begin{aligned}
\|u\|^2_{BMO_2} &\leq 2^{5/4} \sqrt{\|u\|_{BMO_1}} \sqrt{A+B} \\ &\leq
2^{5/4} \sqrt{\|u\|_{BMO_1}}\sqrt{\frac{1}{2}\|u\|^2_{BMO_2} +
\frac{(15e)^2\sqrt{2}}{80} \|u\|^5_{BMO_2}}
\end{aligned}
\]
or rather
\[
\|u\|_{BMO_2} \leq 2^{3/4} \sqrt{\|u\|_{BMO_1}}\sqrt{1 + \frac{5(3e)^2\sqrt{2}}{8}\|u\|^3_{BMO_2}}.
\]
As this inequality is not homogeneous it is helpful to replace $u$
with $cu$ (and subsequently we will optimize over $c$). This yields
\[
\|u\|_{BMO_2} \leq 2^{3/4} \sqrt{\|u\|_{BMO_1}}\sqrt{\frac{1}{c} +
  \frac{c^25(3e)^2\sqrt{2}}{8}\|u\|^3_{BMO_2}}.
\]
We set $D = 5(3e)^2\sqrt{2}\|u\|^3_{BMO_2}/8$, and minimize the
expression under the radical
\[
\frac{1}{c} + Dc^2.
\]
The minimum value is $3(D/4)^{1/3}$. Hence,
\[
\|u\|^2_{BMO_2} \leq 2^{3/2} \|u\|_{BMO_1} 3 (D/4)^{1/3} =\|u\|_{BMO_1}
e^{2/3}5^{1/3}3^{5/3}\|u\|_{BMO_2}
\]
which gives
\[
\|u\|_{BMO_2} \leq
e^{2/3}5^{1/3}3^{5/3}\|u\|_{BMO_1}.
\]
\end{proof}

\end{document}